\documentclass[a4paper,11pt]{article}
\usepackage[latin1]{inputenc}
\usepackage[english]{babel}
\usepackage{amsmath}
\usepackage{amsfonts}
\usepackage{amssymb}
\usepackage{epsfig}
\usepackage{amsopn}
\usepackage{amsthm}
\usepackage{color}
\usepackage{graphicx}
\usepackage{subfigure}
\usepackage{enumerate}
\newtheorem{theorem}{Theorem}[section]

\newtheorem*{theorem*}{Theorem}
\newtheorem*{lemma*}{Lemma}
\newtheorem*{remark*}{Remark}
\newtheorem*{definition*}{Definition}
\newtheorem*{proposition*}{Proposition}
\newtheorem*{corollary*}{Corollary}
\numberwithin{equation}{section}
\newcommand{\real}{\mathbb{R}}

\let\ced=\c         

\def\qed{\,\unskip\kern 6pt \penalty 500
\raise -2pt\hbox{\vrule \vbox to8pt{\hrule width 6pt
\vfill\hrule}\vrule}\par}
\definecolor{darkblue}{rgb}{0.05, .05, .65}
\definecolor{darkgreen}{rgb}{0.1, .65, .1}
\definecolor{darkred}{rgb}{0.8,0,0}
\newcommand{\beqn}{\begin{equation}}
\newcommand{\eeqn}{\end{equation}}
\newcommand{\bear}{\begin{eqnarray}}
\newcommand{\eear}{\end{eqnarray}}
\newcommand{\bean}{\begin{eqnarray*}}
\newcommand{\eean}{\end{eqnarray*}}
\begin{document}

\title{\huge \bf A new transformation for the subcritical fast diffusion equation with source and applications}

\author{
\Large Razvan Gabriel Iagar\,\footnote{Departamento de Matem\'{a}tica
Aplicada, Ciencia e Ingenieria de los Materiales y Tecnologia
Electr\'onica, Universidad Rey Juan Carlos, M\'{o}stoles,
28933, Madrid, Spain, \textit{e-mail:} razvan.iagar@urjc.es},\\
[4pt] \Large Ariel S\'{a}nchez,\footnote{Departamento de Matem\'{a}tica
Aplicada, Ciencia e Ingenieria de los Materiales y Tecnologia
Electr\'onica, Universidad Rey Juan Carlos, M\'{o}stoles,
28933, Madrid, Spain, \textit{e-mail:} ariel.sanchez@urjc.es}\\
[4pt] }
\date{}
\maketitle

\begin{abstract}
A new transformation for radially symmetric solutions to the subcritical fast diffusion equation with spatially inhomogeneous source
$$
\partial_tu=\Delta u^m+|x|^{\sigma}u^p,
$$
posed for $(x,t)\in\real^N\times(0,\infty)$ and with dimension and exponents
$$
N\geq3, \quad 0<m<m_c:=\frac{N-2}{N}, \quad \sigma\in(-2,\infty),
$$
is introduced. It plays a role of a kind of symmetry with respect to the critical exponents
$$
m_s=\frac{N-2}{N+2}, \quad p_L(\sigma)=1+\frac{\sigma(1-m)}{2}, \quad p_s(\sigma)=\frac{m(N+2\sigma+2)}{N-2}.
$$ 
This transformation is then applied for classifying self-similar solutions with or without finite time blow-up to the subcritical fast diffusion equation with source when $p>\max\{1,p_L(\sigma)\}$, having as starting point previous results by the authors. 
\end{abstract}

\


\smallskip


\section{Introduction}

The goal of this paper is to introduce a new \emph{self-map} acting at the level of radially symmetric solutions to the fast diffusion equation with spatially inhomogeneous source
\begin{equation}\label{eq1}
\partial_tu=\Delta u^m+|x|^{\sigma}u^p, \quad (x,t)\in\real^N\times(0,T), \quad T>0,
\end{equation}
posed in dimension $N\geq3$ and in the range of exponents
\begin{equation}\label{range.exp}
m\in\left(0,\frac{N-2}{N}\right), \quad \sigma\in(-2,\infty), \quad p>m
\end{equation}
By self-map we understand a transformation mapping radially symmetric solutions to Eq. \eqref{eq1} with some set of parameters $(m,N,p,\sigma)$ into radially symmetric solutions to the same equation with a different set of parameters (depending on the first ones). We then apply our self-map to deduce a classification of solutions in self-similar form to Eq. \eqref{eq1}, presenting or not finite time blow-up, under some conditions more restrictive than \eqref{range.exp} that will be specified later, starting from recent previous results by the authors.

The fast diffusion equation
\begin{equation}\label{FDE}
u_t=\Delta u^m, \qquad 0<m<1,
\end{equation}
is one of the most important and analyzed models in the theory of nonlinear diffusion equations, both due to its numerous applications and to its quite unexpected mathematical features. A number of applications of fast diffusion in other sciences and engineering such as, for example, gas kinetics, thin liquid film dynamics or anomalous diffusion of hydrogen plasma across a magnetic field, are mentioned in the Introduction of the recent paper by Bonforte and Figalli \cite{BF24} to which we refer (see also references therein), while a good exposition of the mathematical theory can be found in the monograph \cite{VazSmooth} and in the above mentioned work \cite{BF24}.

The mathematical analysis of the fast diffusion equation depends on two fundamental exponents
\begin{equation}\label{crit.exp}
m_c=\frac{(N-2)_{+}}{N}, \qquad m_s=\frac{(N-2)_+}{N+2},
\end{equation}
usually called the \emph{critical exponent} and the \emph{Sobolev exponent}. The former splits the interval $(0,1)$ into the \emph{supercritical range} $m_c<m<1$, where \eqref{FDE} is conservative, preserving the $L^1$ norm of the initial condition along the evolution, and the \emph{subcritical range} $0<m<m_c$, where the phenomenon of finite time extinction is characteristic due to a loss of mass at infinity, a rather unexpected feature explained in \cite[Section 5.5]{VazSmooth}. The latter exponent has a significant influence inside the subcritical range and is closely connected with the Yamabe problem in conformal geometry (see \cite{VazSmooth}). Another major novelty introduced by the subcritical range is the existence of a celebrated branch of self-similar solutions to Eq. \eqref{FDE} known as \emph{anomalous solutions}, see \cite{Ki93, PZ95}, having the following form and behavior at infinity
\begin{equation}\label{anomalous}
u(x,t)=(T-t)^{\alpha}f(|x|(T-t)^{\beta}), \quad f(\xi)\sim C\xi^{-(N-2)/m}, \quad {\rm as} \ \xi\to\infty,
\end{equation}
where the exponents $\alpha$ and $\beta$ and profile $f$ are obtained through an analysis employing dynamical systems techniques (see \cite[Section 7.2]{VazSmooth} for a survey of this theory). These solutions model the finite time extinction in the subcritical range, see \cite{GP97, dPS01}.

Coming back to Eq. \eqref{eq1}, its main mathematical interest is focused on the effect of the competition between the diffusion term and the source term, which tends to increase the $L^1$ norm of a solution, while, as explained above, in the subcritical case $m\in(0,m_c)$, the diffusion alone loses mass and tends to vanish in a finite time. With this motivation in mind, the authors, as part of a larger project of understanding the effect of the presence of either unbounded (at infinity) or singular weights on the dynamical properties of nonlinear diffusion equations, decided to explore this competition and focus on the mathematical properties of Eq. \eqref{eq1}. Previous experience in the study of Eq. \eqref{eq1} in the range \eqref{range.exp} or its slow diffusion counterpart $m>1$ led to the establishment of several critical exponents related to the (inhomogeneous) source term. On the one hand,
\begin{equation}\label{crit.exp.Fujita}
p_L(\sigma)=1+\frac{\sigma(1-m)}{2}, \quad p_F(\sigma)=m+\frac{\sigma+2}{N},
\end{equation}
are related to the blow-up properties of the solutions; indeed, it has been shown that, if $p\leq p_L(\sigma)$, either with $m>1$ in \cite{IMS23} or with $m<1$ in \cite{IS22c, IMS23b}, finite time blow-up is not expected to take place, according to the behavior of the available self-similar solutions in the corresponding range. The exponent $p_F(\sigma)$ is known as the \emph{Fujita type exponent}, separating (if $m\in(m_c,\infty)$) between the range $1<p\leq p_F(\sigma)$ where all the solutions blow up in finite time and the complementary range where there are global solutions, see \cite{Qi98, Su02}. On the other hand, in dimension $N\geq3$, the following critical exponents
\begin{equation}\label{crit.exp.Sobolev}
p_c(\sigma)=\frac{m(N+\sigma)}{N-2}, \quad p_s(\sigma)=\frac{m(N+2\sigma+2)}{N-2}
\end{equation}
are bifurcation exponents for the behavior of self-similar solutions and for other properties of general solutions, as seen in \cite{FT00, IS25, GV97, S4}. We also set
\begin{equation}\label{const.L}
L:=\sigma(m-1)+2(p-1),
\end{equation}
noticing that $p=p_L(\sigma)$ is equivalent to $L=0$.

In our precedent works related to Eq. \eqref{eq1} with exponents as in \eqref{range.exp}, we have established in \cite{IS22c} a new branch of anomalous self-similar solutions with exponential form in the subcritical range $m\in(0,m_c)$ and with $p=p_L(\sigma)$; then, we have classified in \cite{IMS23b} the self-similar solutions (either global or vanishing in finite time) for $m\in(0,m_c)$ and $1<p<p_L(\sigma)$, and finally we did the same in the supercritical range $m\in[m_c,1)$ and $p>p_L(\sigma)$ in \cite{IS25c}. Thus, in the present work we complete the panorama by addressing the remaining range not yet studied, that is $m\in(0,m_c)$ but with $p>p_L(\sigma)$, employing a new transformation (self-map) for the radially symmetric version of Eq. \eqref{eq1} which is of fully independent interest.

\medskip

\noindent \textbf{Main results.} As previously explained, we first introduce a new transformation between radially symmetric solutions to Eq. \eqref{eq1}. Indeed, Eq. \eqref{eq1} writes in radial variables $(r,t)$, $r=|x|$, as
\begin{equation}\label{eq1.rad}
u_t=(u^m)_{rr}+\frac{N-1}{r}(u^m)_r+r^{\sigma}u^p,
\end{equation}
which can obviously be generalized, as an independent equation, to any real parameter $N$ instead of the dimension, as $N$ is just a coefficient in \eqref{eq1.rad}. With this convention, which is typical when dealing with radially symmetric variables and/or ordinary differential equations, we have the following
\begin{theorem}\label{th.transf}
Let $N>2$, $m\in(0,m_c)$ and let $u$ be a (classical for $r\in(0,\infty)$) solution to Eq. \eqref{eq1.rad}. Then the function $\overline{u}(\overline{r},t)$ given by
\begin{equation}\label{transf}
\begin{split}
&\overline{u}(\overline{r},t)=\frac{1}{C_1}r^{(N-2)/m}u(r,t), \quad C_1=\left[\frac{4m^2}{(mN-N+2)^2}\right]^{\sigma/L}, \\
&\overline{r}=C_2r^{\theta}, \quad C_2=C_1^{-(p-1)/\sigma}, \quad \theta=\frac{mN-N+2}{2m},
\end{split}
\end{equation}
is a solution to Eq. \eqref{eq1.rad} with independent variable $\overline{r}$ and parameters given by
\begin{equation}\label{transf.exp}
\overline{N}=-\frac{2(N-2m-2)}{mN-N+2}, \quad \overline{\sigma}=-\frac{2[(N-2)(p-1)-m\sigma]}{mN-N+2}.
\end{equation}
\end{theorem}

\noindent \textbf{Remark.} For fixed $p>m$, this transformation can be understood as a \emph{symmetry of the Sobolev exponents} $p_s(\sigma)$ and $p_s(\overline{\sigma})$ with respect to $p$, in the sense that
\begin{equation}\label{symm}
p-p_s(\overline{\sigma})=p_s(\sigma)-p.
\end{equation}
Moreover, the restriction $m\in(0,m_c)$ ensures that the interval $N\in(2,\infty)$ is mapped one-to-one into $\overline{N}\in(2,\infty)$.

Let us also observe that, in our range of interest $m\in(0,m_c)$, the transformation \eqref{transf} is an inversion, since $\theta<0$ in this range. Moreover, one can readily check, among other properties of this transformation, that it maps the interval $m\in(m_s,m_c)$ with respect to dimension $N$ in Eq. \eqref{eq1.rad} into the interval $m\in(0,m_s)$ with respect to dimension parameter $\overline{N}$ in Eq. \eqref{eq1.rad}. As a precedent, one can notice that the transformation \eqref{transf} is related to the ones introduced for the non-homogeneous porous medium equation in \cite{IRS13} and even more with the second transformation in \cite[Section 2.2]{IS23} for the particular case $\sigma_1=0$ in the notation therein, where a similar change of variable is noticed but discarded as not very useful in the slow diffusion range $m>1$. It appears that, in change, this inversion works very well for the subcritical fast diffusion range.

\medskip

\noindent \textbf{Application to self-similar solutions.} We apply this transformation, together with the results in \cite{IMS23b}, to complete the classification of the self-similar solutions to Eq. \eqref{eq1} in the subcritical case started in \cite{IS22c, IMS23b}, gathering both global solutions in the form
\begin{equation}\label{forward.SS}
u(x,t)=t^{-\alpha}f(\xi), \qquad \xi=|x|t^{-\beta},
\end{equation}
and solutions with finite time blow-up in the form
\begin{equation}\label{backward.SS}
u(x,t)=(T-t)^{-\alpha}f(\xi), \qquad \xi=|x|(T-t)^{-\beta}, \qquad T\in(0,\infty),
\end{equation}
where in both cases the self-similar exponents are given by
\begin{equation}\label{SS.exp}
\alpha=\frac{\sigma+2}{L}>0, \qquad \beta=\frac{p-m}{L}>0.
\end{equation}
By replacing the ansatz \eqref{forward.SS}, respectively \eqref{backward.SS} into Eq. \eqref{eq1}, we find by direct calculation that the profiles $f(\xi)$ of the self-similar solutions solve the following differential equations:
\begin{equation}\label{ODE.forward}
(f^m)''(\xi)+\frac{N-1}{\xi}(f^m)'(\xi)+\alpha f(\xi)+\beta\xi f'(\xi)+\xi^{\sigma}f^p(\xi)=0,
\end{equation}
in the case of self-similar solutions in the form \eqref{forward.SS}, or
\begin{equation}\label{ODE.backward}
(f^m)''(\xi)+\frac{N-1}{\xi}(f^m)'(\xi)-\alpha f(\xi)-\beta\xi f'(\xi)+\xi^{\sigma}f^p(\xi)=0,
\end{equation}
in the case of self-similar solutions in the form \eqref{backward.SS}. The next theorem gathers all the ranges of existence and non-existence in the range of exponents specified in the statement.
\begin{theorem}\label{th.subcrit}
Let $N\geq3$, $m\in(0,m_c)$, $\sigma>-2$ and $p>\max\{p_L(\sigma),1\}$.
\begin{enumerate}
  \item For any $m\in(m_s,m_c)$ and $p\in(p_L(\sigma),p_s(\sigma))$ with $p>1$, there exists at least one global self-similar solution to Eq. \eqref{eq1} in the form \eqref{forward.SS}, with the fast decay \eqref{anomalous} as $\xi\to\infty$. Moreover, in the same range there are infinitely many self-similar solutions presenting the slow decay
\begin{equation}\label{beh.Q1f}
f(\xi)\sim K\xi^{-(\sigma+2)/(p-m)}, \qquad {\rm as} \ \xi\to\infty, \qquad K>0.
\end{equation}
There are no global in time self-similar solutions for $p>p_s(\sigma)$.
  \item For any $m\in(m_s,m_c)$, there exist $p_0(\sigma)>p_s(\sigma)$ such that, for any $p\in(p_s(\sigma),p_0(\sigma))$ there exists at least one self-similar solution in the form \eqref{backward.SS}, presenting finite time blow-up with the fast decay rate \eqref{anomalous} as $\xi\to\infty$. Moreover, there exists $p_1(\sigma)\geq p_0(\sigma)$ such that for $p=p_1(\sigma)$, there are self-similar solutions in the form \eqref{backward.SS} with the behavior
\begin{equation}\label{beh.P3}
f(\xi)\sim\left[\frac{2m(mN-N+2)}{1-m}\right]^{1/(1-m)}\xi^{-2/(1-m)}, \qquad {\rm as} \ \xi\to\infty.
\end{equation}
Finally, there exists $p_2(\sigma)\geq p_1(\sigma)$ such that for $p>p_2(\sigma)$, there is no self-similar solution in the form \eqref{backward.SS} to Eq. \eqref{eq1}.
  \item For any $m\in(0,m_s]$, there are no self-similar solutions to Eq. \eqref{eq1} with any possible decay as $|x|\to\infty$, in any of the two forms \eqref{forward.SS} or \eqref{backward.SS}.
\end{enumerate}
\end{theorem}

\medskip

\noindent \textbf{Remarks. 1}. The order of the critical exponents is implicitly understood in the statements of the theorems. Indeed, in the range of main interest $m\in(m_s,m_c)$ (as seen in Theorem \ref{th.subcrit}) we have $p_c(\sigma)<p_L(\sigma)<p_s(\sigma)$, as one can readily check. Moreover, $p_L(\sigma)>1$ is equivalent to $\sigma>0$, while $p_s(\sigma)>1$ holds true always if $m>m_s$ and $\sigma\geq0$.

\textbf{2}. All our results hold true and are new for $\sigma=0$, which is just a particular case in the analysis.

\textbf{3. Regularity.} A similar discussion related to the regularity at $\xi=0$ as in \cite{IS25c} applies here. In particular, while for $\sigma\geq0$ the radially symmetric self-similar solutions we obtain are of class $C^2(\real^N)$ and thus classical solutions, for $\sigma<0$ they can be of class $C^1$ (if $\sigma\in(-1,0)$) or even only $C^{0,\gamma}$ for some $\gamma\in(0,1)$ (if $\sigma\in(-2,-1)$), forming a peak as noticed in rather similar situations \cite{RV06, IL25}. We expect that the optimal regularity at $\xi=0$ as in \cite[Section 3.3]{IL25} still holds true for our self-similar solutions.

We are now ready to give the proofs of the two theorems stated as main results of this paper.

\section{Proof of Theorem \ref{th.transf}}\label{sec.transf}

This section is devoted to the proof of Theorem \ref{th.transf}, thus we assume throughout it that $N>2$ and $0<m<m_c$. Moreover, it is straightforward to check that
\begin{equation}\label{interm1.bis}
\begin{split}
&p_c(\sigma)<p_L(\sigma) \qquad {\rm iff} \qquad m\in(0,m_c), \\
&p_s(\sigma)<p_L(\sigma) \qquad {\rm iff} \qquad m\in(0,m_s),
\end{split}
\end{equation}
hence we have by default $p>p_c(\sigma)$ in this range. We show below that the transformation \eqref{transf} works as claimed. Let us fix from the beginning that, since we only deal with equations expressed in radially symmetric variables, we will make the convention that the dimensions appearing in the calculations (that is, $N$ and $\overline{N}$) will be considered as real parameters, as they just appear as coefficients in the differential equations in variables $(r,t)$ (respectively $(\overline{r},t)$).
\begin{proof}[Proof of Theorem \ref{th.transf}]
Of course, one can say that the proof follows from direct calculation. But in order to be honest with the reader, we will show how we obtained the transformation. Trying to generalize to our equation the transformation in \cite[Case 2, Section 2.1]{IRS13} and thus taking the first exponent from there, we plug in the radial equation \eqref{eq1.rad} the ansatz
\begin{equation}\label{interm20}
\overline{u}(\overline{r},t)=\frac{1}{C_1}r^{(N-2)/m}u(r,t), \qquad \overline{r}=C_2r^{\theta}, \qquad r\in(0,\infty),
\end{equation}
with $C_1$, $C_2$, $\theta$ to be determined. By replacing the ansatz \eqref{interm20} into Eq. \eqref{eq1.rad}, we get
\begin{equation}\label{interm21}
\begin{split}
\overline{u}_t&=\theta^2C_1^{m-1}C_2^2r^{-(N-2)(m-1)/m+2\theta-2}(\overline{u}^m)_{\overline{r}\overline{r}}\\&+\theta C_1^{m-1}C_2^2(\theta-N+2)r^{-(N-2)(m-1)/m+\theta-2}(\overline{u}^m)_{\overline{r}}\\
&+C_1^{p-1}C_2^{\sigma}r^{\sigma-(N-2)(p-1)/m}\overline{u}^p.
\end{split}
\end{equation}
We next impose the condition that the coefficient of the first term on the right hand side to be equal to one, and the last term in the right hand side should have no constant in front, in order to match \eqref{interm21} to Eq. \eqref{eq1.rad}. We thus get the equalities:
$$
2\theta-\frac{(N-2)(m-1)}{m}-2=0, \qquad C_1^{m-1}C_2=\frac{1}{\theta^2}, \qquad C_1^{p-1}C_2^{\sigma}=1,
$$
from which we readily derive that $\theta=(mN-N+2)/2m<0$ (since $m\in(0,m_c)$) and the expressions of $C_1$, $C_2$ in \eqref{transf}. Moreover, by further identifying \eqref{interm21} to Eq. \eqref{eq1.rad} in variables $(\overline{u},\overline{r})$ we also obtain that
$$
\overline{\sigma}=\frac{1}{\theta}\left[\sigma-\frac{(N-2)(p-1)}{m}\right], \qquad \overline{N}-1=\frac{\theta-N+2}{\theta},
$$
which lead to the expressions of $\overline{\sigma}$, $\overline{N}$ given in \eqref{transf.exp}. All the previous calculations are valid for classical solutions in $r\in(0,\infty)$, but can be extended to weak solutions at $r=0$ similarly as in \cite[Section 4]{IRS13}, since $\theta<0$. The local behavior of $\overline{u}$ at $\overline{r}=0$ is then determined by the behavior of $u$ as $r\to\infty$, and the proof is complete.
\end{proof}
Let us notice a few facts related to the parameters appearing in Eq. \eqref{eq1.rad}. First, if $m<m_c$ and $N>2$, we observe that
\begin{equation}\label{interm22}
\overline{N}-2=-\frac{2m(N-2)}{mN-N+2}>0,
\end{equation}
hence $\overline{N}>2$. In particular, it follows that the interval $(2,\infty)$ is mapped onto itself by the transformation from $N$ to $\overline{N}$. Furthermore, an easy calculation leads to
\begin{equation}\label{interm23}
m-\frac{\overline{N}-2}{\overline{N}+2}=\frac{m(m-m_s)}{m-2m_s}<0,
\end{equation}
provided $m\in(m_s,2m_s)$. Since
$$
2m_s-m_c=\frac{(N-2)^2}{N(N+2)}>0,
$$
we deduce that, if $m\in(m_s,m_c)$, we obtain that $m<(\overline{N}-2)/(\overline{N}+2)=m_s(\overline{N})$, while if $m\in(0,m_s)$, then $m>m_s(\overline{N})$. We furthermore remark that
\begin{equation}\label{interm24}
\overline{\sigma}-\frac{2(p-1)}{1-m}=-\frac{2mL}{(1-m)(mN-N+2)}>0,
\end{equation}
whence $\overline{\sigma}>\sigma_{L}:=2(p-1)/(1-m)$, provided $L>0$ (that is, $p>p_L(\sigma)$ or equivalently $\sigma<\sigma_L$). Thus, we match the condition $p>p_L(\sigma)$ into the condition $p<p_L(\overline{\sigma})$. One more step is to observe that in our range, we always have $p>p_c(\overline{\sigma})$, since
\begin{equation}\label{interm25}
p-p_c(\overline{\sigma})=p-\frac{m(\overline{N}+\overline{\sigma})}{\overline{N}-2}=\frac{m(\sigma+2)}{N-2}>0,
\end{equation}
as we are always under the condition $N\geq3$ (or $N>2$ as a real parameter with the convention we did). Finally, the symmetry \eqref{symm} of $p_s(\sigma)$ and $p_s(\overline{\sigma})$ with respect to $p$ follows by noticing, after straightforward calculations employing \eqref{transf.exp}, that
$$
\overline{N}+2\overline{\sigma}+2=\frac{2(mN+2m+2m\sigma+4p-2Np)}{mN-N+2},
$$
which, together with \eqref{interm22}, gives
\begin{equation*}
p-p_s(\overline{\sigma})=p-\frac{m(\overline{N}+2\overline{\sigma}+2)}{\overline{N}-2}==\frac{m(N+2\sigma+2)-p(N-2)}{N-2}=p_s(\sigma)-p.
\end{equation*}
All these properties will be of a great use in the proof of Theorem \ref{th.subcrit}, given in the next section.

\section{Proof of Theorem \ref{th.subcrit}}\label{sec.subcrit}

Since self-similar solutions in the form \eqref{forward.SS} or \eqref{backward.SS} are particular cases of radially symmetric solutions to Eq. \eqref{eq1}, we are now in a position to complete the proof of Theorem \ref{th.subcrit} by employing the transformation \eqref{transf} and all the previous connections between the sets of parameters of Eq. \eqref{eq1.rad} before and after the transformation.
\begin{proof}[Proof of Theorem \ref{th.subcrit}]
Gathering \eqref{interm22}, \eqref{interm23}, \eqref{interm24} and \eqref{symm}, we conclude that we are in a position to apply the results in \cite{IMS23b} for the equation in variables $(\overline{r},t)$ and exponents $\overline{N}$, $\overline{\sigma}$, $p$ and $m$. Indeed, an easy inspection of the proofs in \cite{IMS23b} shows that, if considering $\overline{N}$ as a real parameter in the equations of the dynamical systems, all the results therein hold true for $\overline{N}>2$, thus can be mapped to the interval $N>2$ (and in particular for integer dimensions $N\geq3$) in our case, according to \eqref{interm22}. Then, if $m\in(0,m_s)$, we deduce 
$$
m>m_s(\overline{N}):=\frac{\overline{N}-2}{\overline{N}+2}
$$ 
and no self-similar solutions exist according to \cite{IMS23b}. We are left only with the range $m\in(m_s,m_c)$ and \eqref{interm25} ensures that $p>p_c(\overline{\sigma})$, exactly the condition leading to existence of solutions in \cite{IMS23b}.

In order to apply the transformation for radially symmetric self-similar solutions, we have next to check how it works on profiles. On the one hand, it is rather obvious to notice that \eqref{transf} preserves the form of the self-similar solution. Let us assume first that $\overline{u}(\overline{r},t)$ is a self-similar solution in backward form \eqref{backward.SS} with a profile $\overline{f}$. Noticing that in bar variables we are in the range $p<p_L(\overline{\sigma})$, according to \eqref{interm24}, we obtain from the analysis in \cite{IMS23b} that the self-similar exponents are given by the same expressions as in \eqref{SS.exp} but with a minus sign in front. Let then $\overline{\alpha}$, $\overline{\beta}$ be these self-similar exponents in bar variables. We have:
\begin{equation*}
\begin{split}
u(r,t)&=C_1r^{-(N-2)/m}(T-t)^{\overline{\alpha}}\,\overline{f}(\overline{r}(T-t)^{\overline{\beta}})
=C_1r^{-(N-2)/m}(T-t)^{\overline{\alpha}}\,\overline{f}(C_2r^{\theta}(T-t)^{\overline{\beta}})\\
&=C_1r^{-(N-2)/m}(T-t)^{\overline{\alpha}}\,\overline{f}\left(C_2(r(T-t)^{\overline{\beta}/\theta})^{\theta}\right)\\
&=C_1(r(T-t)^{\overline{\beta}/\theta})^{-(N-2)/m}(T-t)^{\overline{\alpha}+(N-2)\overline{\beta}/m\theta}\overline{f}\left(C_2(r(T-t)^{\overline{\beta}/\theta})^{\theta}\right)\\
&=C_1\xi^{-(N-2)/m}(T-t)^{-\alpha}\overline{f}(C_2(r(T-t)^{-\beta})^{\theta}),
\end{split}
\end{equation*}
where $\alpha$ and $\beta$ are given by \eqref{SS.exp}. We then deduce the following correspondence:
\begin{equation}\label{transf.SS}
f(\xi)=C_1\xi^{-(N-2)/m}\overline{f}(\overline{\xi}), \qquad \overline{\xi}=C_2\xi^{\theta},
\end{equation}
where $C_1$, $C_2$, $\theta$ are defined in \eqref{transf} and \eqref{transf.exp}. Notice that the previous calculations, due to the change of signs in the exponents $\alpha$ and $\beta$ with respect to the starting ones $\overline{\alpha}$ and $\overline{\beta}$, prove that self-similar solutions presenting \emph{finite time extinction} in bar variables are mapped into self-similar solutions presenting \emph{finite time blow-up} in original variables. A completely similar calculation to the previous one gives the same correspondence \eqref{transf.SS} also when changing global solutions in bar variables into global solutions in the original variables. Let us finally remark that $\theta<0$ (since $m\in(0,m_c)$), thus the transformation is an inversion, reversing the local behaviors from $\overline{\xi}\to\infty$ into $\xi\to0$ and viceversa. We give next the map of changes obtained from \eqref{transf.SS} with respect to the (interesting) local behaviors of profiles:

$\bullet$ if $\overline{f}(\overline{\xi})\sim C\overline{\xi}^{-(\overline{N}-2)/m}$ as $\overline{\xi}\to\infty$, then
$$
f(\xi)\sim C\xi^{-(N-2)/m-(\overline{N}-2)\theta/m}=C, \qquad {\rm as} \ \xi\to0.
$$

$\bullet$ if $\overline{f}\sim C$ as $\overline{\xi}\to0$, then obviously $f(\xi)\sim C\xi^{-(N-2)/m}$ as $\xi\to\infty$.

$\bullet$ if $\overline{f}\sim C\overline{\xi}^{-(\overline{\sigma}+2)/(p-m)}$ as $\overline{\xi}\to\infty$, it follows that
$$
f(\xi)\sim C\xi^{-(N-2)/m-\theta(\overline{\sigma}+2)/(p-m)}=C\xi^{-(\sigma+2)/(p-m)}, \qquad {\rm as} \ \xi\to0,
$$
and viceversa, if the first behavior is taken as $\overline{\xi}\to0$, then the similar one in original variables is taken as $\xi\to\infty$.

$\bullet$ if $\overline{f}\sim C\overline{\xi}^{-2/(1-m)}$ as $\overline{\xi}\to0$, then $f(\xi)\sim C\xi^{-2/(1-m)}$ as $\xi\to\infty$.

It only remains to apply this list of changes to the profiles (with corresponding local behaviors) obtained in \cite{IMS23b}. We first establish the \textbf{global self-similar solutions}, which are mapped into global self-similar solutions given by \cite[Theorem 1.1]{IMS23b} in bar variables. Indeed, the quoted theorem states that, if $\max\{1,p_s(\overline{\sigma})\}<p<p_L(\overline{\sigma})$, there exist global self-similar solutions whose profiles satisfy $f(0)=C>0$ and $\overline{f}(\overline{\xi})\sim \overline{\xi}^{-(\overline{N}-2)/m}$ as $\overline{\xi}\to\infty$. We infer then from \eqref{symm}, \eqref{interm24} and the above list of correspondences obtained as a consequence of \eqref{transf.SS} that there are profiles with local behavior as $\xi\to0$ given by
\begin{equation}\label{beh.P0f}
f(\xi)\sim\left\{\begin{array}{ll}\left[D+\frac{\alpha(1-m)}{2mN}\xi^2\right]^{-1/(1-m)}, & \sigma>0,\\[1mm]
\left[D+\frac{(1-m)\alpha(1+\alpha D^{(p-1)/(m-1)})}{2mN}\xi^2\right]^{-1/(1-m)}, & \sigma=0,\\[1mm]
\left[D+\frac{p-m}{m(N+\sigma)(\sigma+2)}\xi^{\sigma+2}\right]^{-1/(p-m)}, & \sigma\in(-2,0),
\end{array}\right. \qquad D>0,
\end{equation}
and as $\xi\to\infty$ given by \eqref{anomalous} if and only if $p_L(\sigma)<p<p_s(\sigma)$. Moreover, an inspection of the proof of \cite[Theorem 1.1]{IMS23b} shows that there are infinitely many orbits (those corresponding to the nonempty open set $\mathcal{A}$ in the notation therein) that have the local behavior \eqref{beh.Q1f} as $\overline{\xi}\to0$ and the local behavior \eqref{anomalous} as $\overline{\xi}\to\infty$. By \eqref{transf.SS}, these profiles are mapped onto profiles with local behavior \eqref{beh.P0f} but with the slow decay \eqref{beh.Q1f} as $\xi\to\infty$. We have thus completed the proof of the first item in Theorem \ref{th.subcrit}.

With respect to \textbf{self-similar solutions with finite time blow-up}, we have to apply the transformation \eqref{transf.SS} to the outcome of \cite[Theorem 1.2]{IMS23b}. On the one hand, solutions with the fast decay \eqref{anomalous} are shown to exist for $p\in(p_0(\overline{\sigma}),p_s(\overline{\sigma}))$ for some $p_0(\overline{\sigma})>p_c(\overline{\sigma})$. This interval is mapped, via \eqref{symm}, onto an interval $(p_s(\sigma),p_0(\sigma))$. On the other hand, the remark at the end of \cite[Section 8]{IMS23b} gives the existence of at least one exponent $p_*(\overline{\sigma})\in(p_c(\overline{\sigma}),p_s(\overline{\sigma}))$ such that for $p=p_*(\overline{\sigma})$, there exists one self-similar profile with local behaviors
\begin{equation*}
\overline{f}(\overline{\xi})\sim C\overline{\xi}^{-2/(1-m)}, \qquad {\rm as} \ \overline{\xi}\to0, \qquad
\overline{f}(\overline{\xi})\sim C\overline{\xi}^{-(\overline{N}-2)/m}, \qquad {\rm as} \ \overline{\xi}\to\infty.
\end{equation*}
The list of correspondences above between local behaviors, together with \eqref{interm25} and \eqref{symm}, ensure then the existence of the exponent $p_1(\sigma)\geq p_0(\sigma)$ for which there exists a profile with local behavior as $\xi\to0$ given by
\begin{equation}\label{beh.P0b}
f(\xi)\sim\left\{\begin{array}{ll}\left[D-\frac{\alpha(1-m)}{2mN}\xi^2\right]^{-1/(1-m)}, & \sigma>0,\\[1mm]
\left[D-\frac{(1-m)\alpha(1-\alpha D^{(p-1)/(m-1)})}{2mN}\xi^2\right]^{-1/(1-m)}, & \sigma=0,\\[1mm]
\left[D+\frac{p-m}{m(N+\sigma)(\sigma+2)}\xi^{\sigma+2}\right]^{-1/(p-m)}, & \sigma<0,\end{array}\right., \quad D>0,
\end{equation}
and with decay as $\xi\to\infty$ given by \eqref{beh.P3}. The proof of the second item in Theorem \ref{th.subcrit} is then completed by the non-existence range in \cite[Theorem 1.2]{IMS23b}, and the same happens with the third and last item in Theorem \ref{th.subcrit}, which follows from corresponding non-existence ranges in \cite[Theorem 1.1 and Theorem 1.2]{IMS23b}. We omit the details of these last non-existence results, as they are straightforward from the transformation.
\end{proof}
We complete this section with a discussion of the \emph{evolution of hot spots} (that is, maximum points) of the self-similar solutions classified in Theorem \ref{th.subcrit}, and in particular, of their \emph{blow-up set} (when finite time blow-up takes place). We first notice from results in \cite{IMS23b} and our transformation that the profiles of the global self-similar solutions are decreasing and have a maximum point at $\xi=0$. It follows that the corresponding self-similar solutions decay to zero as $t\to\infty$ as follows:
\begin{equation}\label{sol.forward}
\|u(t)\|_{\infty}=t^{-\alpha}f(0), \qquad t>0,
\end{equation}
and the same occurs for the evolution of a fixed point $|x|=r>0$, since $\xi(t)=|x|t^{-\beta}\to0$ as $t\to\infty$. On the contrary, the profiles of the solutions with finite time blow-up with $\sigma>0$ are increasing in a right neighborhood of $\xi=0$ and attain their maximum at some point $\xi_0>0$. Then
\begin{equation}\label{sol.backward}
\|u(t)\|_{\infty}=(T-t)^{-\alpha}f(\xi_0)\to\infty, \qquad {\rm as} \ t\to T,
\end{equation}
hence these profiles have a blow-up rate $(T-t)^{-\alpha}$, and their maximum at time $t\in(0,T)$ is attained for $|x|=(T-t)^{\beta}\xi_0\to0$ as $t\to T$. More interestingly, given a fixed point $|x|=r>0$, we have, for our specific tails \eqref{anomalous}, \eqref{beh.Q1f} or \eqref{beh.P3} and as $t\to T$,
$$
u(x,t)\sim (T-t)^{-\alpha}(|x|(T-t)^{-\beta})^{\gamma}=|x|^{\gamma}(T-t)^{-\alpha-\beta\gamma}<\infty,
$$
where
$$
\gamma\in\left\{-\frac{N-2}{m},-\frac{2}{1-m},-\frac{\sigma+2}{p-m}\right\}
$$
is the decay exponent as $\xi\to\infty$ of the profile of the solution. This follows by direct calculation, since for any of the three possible values of $\gamma$ we have $-\alpha-\beta\gamma\geq0$. Similar considerations with respect to the evolution of a fixed point hold true also for the blow-up solutions with $\sigma<0$, which are decreasing. We thus find that \emph{the blow-up set} of all the self-similar solutions with finite time blow-up \emph{is the singleton $\{0\}$}.

\section{Conclusion. A final discussion about self-similar solutions to Eq. \eqref{eq1}}

In this final section, we gather in form of a conclusion the most interesting fact that springs out of our papers \cite{IS22c, IMS23b, IS25c} and the current one: the enormous \emph{diversity and richness of phenomena} that the fast diffusion with spatially inhomogeneous source involves. In particular, we have noticed in these works that, for $m\in(0,m_s)$ we can only have self-similar solutions presenting finite time extinction or grow-up (that is, global solutions unbounded as $t\to\infty$), see \cite{IMS23b}, while for $m\in(m_s,1)$ we can only have solutions with finite time blow-up or global solutions decaying to zero as $t\to\infty$, see the present work and \cite{IS25c}. Let us observe that, for $\sigma=0$, we are always in the range $p>p_L(\sigma)$ and thus self-similar solutions with finite time extinction or with grow-up as $t\to\infty$ do not exist, but they start to exist, as detailed in Figure \ref{fig1} below, for any $\sigma>0$ and $1<p<p_L(\sigma)$. This is an effect of the presence of the variable coefficient $|x|^{\sigma}$, with respect to the homogeneous equation with $\sigma=0$.

For a complete classification, we detail below all these different possible dynamical properties observed at the level of self-similar solutions, plotting in Figure \ref{fig1} the regions in which they occur (limited by the critical exponents in \eqref{crit.exp}, \eqref{crit.exp.Fujita} and \eqref{crit.exp.Sobolev}), for generic $\sigma\geq0$, depending on the ranges of $m\in(0,1)$ and $p>1$.

\begin{figure}[ht!]
  \begin{center}
  \includegraphics[width=12cm,height=12cm]{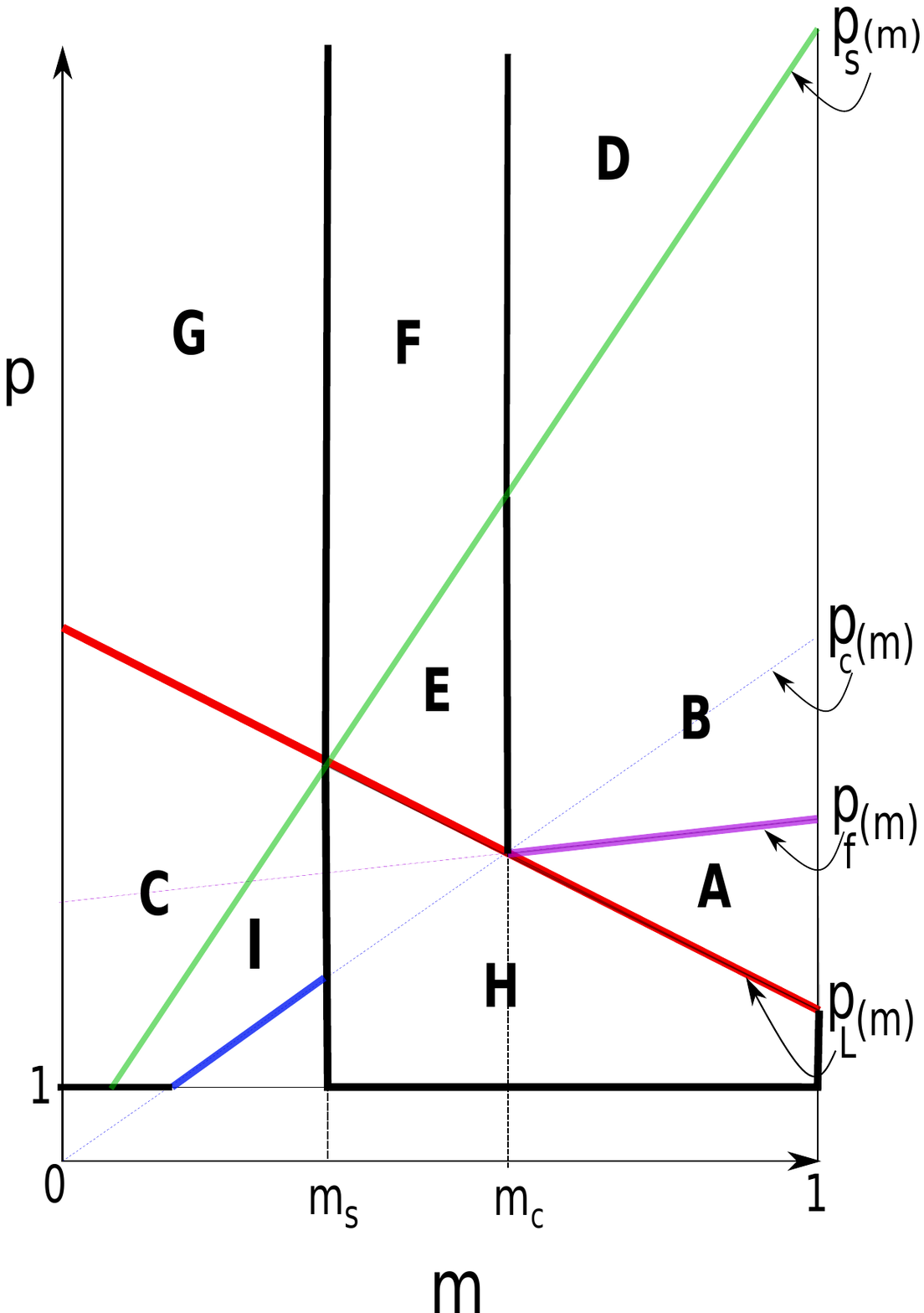}
  \end{center}
  \caption{Regions of the plane $(p,m)$ with different behavior for self-similar solutions, for a generic $\sigma\geq0$.}\label{fig1}
\end{figure}

Indeed, we have established in the above mentioned papers that self-similar solutions to Eq. \eqref{eq1}, in dependence of the exponents of the equation, present the following different dynamical behaviors as time advances:

$\bullet$ solutions with \emph{finite time blow-up}: an outcome of \cite{IS25c} and Theorem \ref{th.subcrit} (in the range $p>p_s(\sigma)$ if subcritical) in the current work. For $\sigma\geq0$, this range corresponds to the region $\mathbf{F}$ in Figure \ref{fig1}, while region $\mathbf{B}$ (and also $\mathbf{A}$ if $\sigma=0$) in Figure \ref{fig1} corresponds to the non-existence range for such solutions.

$\bullet$ global solutions \emph{vanishing as $t\to\infty$}: again an outcome of \cite{IS25c} and Theorem \ref{th.subcrit} in the current paper. For $m\geq m_c$, there are solutions having a specific optimal (fast) decay as $|x|\to\infty$ given by \eqref{beh.P3}, and they exist in the range corresponding to the region $\mathbf{B}$ in Figure \ref{fig1}. Moreover, in the range corresponding to the region $\mathbf{D}$ in Figure \ref{fig1} there are such solutions, but only with a slow decay as $|x|\to\infty$ given by \eqref{beh.Q1f}. In the subcritical range $m<m_c$, there exist such solutions with both optimal (fast) decay \eqref{anomalous} and the slow decay \eqref{beh.Q1f} as $|x|\to\infty$, in the range corresponding to the region $\mathbf{E}$ in Figure \ref{fig1}.

$\bullet$ global solutions \emph{growing up as $t\to\infty$}: they have been established in \cite[Theorem 1.1]{IMS23b} for $m\in(0,m_s)$, $\sigma>0$ and $p_s(\sigma)<p<p_L(\sigma)$. This range corresponds to the region $\mathbf{C}$ in Figure \ref{fig1}.

$\bullet$ solutions with \emph{finite time extinction}: they have been obtained in \cite[Theorem 1.2]{IMS23b}, for $m\in(0,m_s)$ and either $p_c(\sigma)<p<p_s(\sigma)$ but $p$ close to $p_s(\sigma)$, where they have the fast decay \eqref{anomalous} as $|x|\to\infty$, or $p\in(p_s(\sigma),p_L(\sigma))$, where they have the slow decay \eqref{beh.Q1f} as $|x|\to\infty$. The former range corresponds to the region $\mathbf{I}$ in Figure \ref{fig1}, while the latter corresponds again to the region $\mathbf{C}$.

$\bullet$ \emph{eternal solutions} in exponential form: they have been deduced in \cite{IS22c} for the limiting case $p=p_L(\sigma)$. Again, they can have either exponential grow-up as $t\to\infty$ if $m\in(0,m_s)$, or exponential decay to zero as $t\to\infty$ if $m\in(m_s,m_c)$.

$\bullet$ \emph{stationary solutions} for $p=p_s(\sigma)$, in all the cases of $m\in(0,1)$ and $\sigma>-2$, $N\geq3$, with an explicit formula given in \cite{IMS23b, IS25c}.

$\bullet$ \emph{non-existence of any type of self-similar solutions} holds true in the following ranges: $m\in(0,m_s)$ and $p>p_L(\sigma)$, corresponding to the region $\mathbf{G}$ in Figure \ref{fig1}, as established in Theorem \ref{th.subcrit}, $m\in(m_s,1)$ and $1<p<p_L(\sigma)$, corresponding to the region $\mathbf{H}$ in Figure \ref{fig1}, as established in \cite{IMS23b}, and in the range $m\in[m_c,1)$ and (at least a part of) the range $p\in(p_L(\sigma),p_F(\sigma))$, corresponding to the region $\mathbf{A}$ in Figure \ref{fig1}, as established in \cite{IS25c}.

We strongly believe that this bunch of different solutions will be useful in the future to open up the way towards a rigorous functional analytic study of Eq. \eqref{eq1} and of its large time behavior either as $t\to\infty$ (for global solutions) or as $t\to T$ (for solutions with finite time blow-up or finite time extinction).

\bigskip

\noindent \textbf{Acknowledgements} R. G. I. and A. S. are partially supported by the Project PID2020-115273GB-I00 and by the Grant RED2022-134301-T funded by MCIN/AEI/10.13039/ \\ 501100011033 (Spain).

\bibliographystyle{plain}

\end{document}